\documentclass[12pt]{article}
\usepackage{amsmath,amssymb,dsfont,tikz-cd,amsthm}

\title{A Decomposition Theorem for Aronszajn Lines}
\author{Keegan Dasilva Barbosa}
\date{ }

\newcommand{\RR}{\mathbb{R}}

\newcommand{\NN}{\mathbb{N}}

\newcommand{\QQ}{\mathbb{Q}}

\newtheorem{theorem}{Theorem}[section]
\newtheorem{prop}{Proposition}
\newtheorem{corollary}{Corollary}[theorem]
\newtheorem{lemma}[theorem]{Lemma}
\theoremstyle{definition}
\newtheorem{definition}{Definition}[section]
\newtheorem{fact}{Fact}[section]

\begin{document}
\maketitle

\begin{abstract}
	We show that under the proper forcing axiom the class of all Aronszajn lines behave like $\sigma$-scattered orders under the embeddability relation. In particular, we are able to show that the class of better quasi order labeled fragmented Aronszajn lines is itself a better quasi order. Moreover, we show that every better quasi order labeled Aronszajn line can be expressed as a finite sum of labeled types which are algebraically indecomposable. By encoding lines with finite labeled trees, we are also able to deduce a decomposition result, that  for every Aronszajn line $L$ there is integer n such that for any finite colouring of $L$ there is subset $L^\prime$ of $L$ isomorphic to $L$ which uses no more than n colours.
	 
\end{abstract}
\section{Introduction}
It was shown by Carlos Martinez-Ranero that under PFA, the class of Aronszajn lines is a better quasi order under the embeddability relation \cite{Carlos}. The proof requires the development of an analogue to Hausdorff rank for scattered linear orders \cite{Hausdorff} to the Aronszajn case, as well as the construction of incompatible Aronszajn lines $D_\alpha^+ $ and $D_\alpha^-$, $\alpha \in \omega_2 $, that behave as universal lines of rank $\leq\alpha$. What is interesting about the lines $D_\alpha^+$ and $D_\alpha^-$ is that they are fairly homogeneous. In particular, they are algebraically indecomposable and are recursively constructed via a variant of shuffle described by Laver \cite{Laver1}. Laver used shuffles to recursively construct the class of $\sigma$-scattered linear orders and showed the class was BQO. We follow his construction to find analogues of his results in the context of fragmented Aronszajn lines. In particular, we are able to recursively construct a well behaved class of algebraically indecomposable fragmented Aronszajn lines that behave as the building blocks for all fragmented Aronszajn types.
\begin{theorem}(PFA)
	Let $Q$ be a BQO. Any $Q$ labeled fragmented Aronszajn line $\Phi \in \mathcal{C}(Q)$ can be written as a finite sum of members from a class of $Q$ labeled algebraically indecomposable Aronszajn lines $\mathcal{H}(Q)$. 
\end{theorem}
By showing that our constructed class $\mathcal{H}(Q)$ is BQO, we are able to give an alternative proof that the class of Aronszajn lines is BQO under PFA. More interestingly, we are also able to code algebraically indecomposable Aronszajn lines onto BQO labeled trees like in \cite{Laver2} to show the following decomposition theorem for Aronszajn lines. 
\begin{theorem}{(PFA)}
	For any Aronszajn type $\phi$, there is an $n\in \omega$ such that $\forall k\in \omega $ $\phi \rightarrow (\phi)_{k,n}^1 $
\end{theorem}
The paper is organized as follows. Section 2 will be devoted entirely to preliminaries. If one is familiar with linear orders, BQO's and Ramsey degrees, they can skip to section 2.4. While most of 2.4 is stating well known facts about Aronszajn lines, we also define shuffles here which will become highly relevant later. Section 3 is split into three major components. Section 3.1 is devoted to the finer structure analysis i.e defining shuffles and recursively constructing every Aronszajn line via shuffles of $\{0,1\}$. We also prove some relevant properties this class has. In the second, we prove Theorem 1.1 using Laver's techniques from \cite{Laver1}. In the last, we prove theorem 1.2 by adapting Laver's finite tree coding argument \cite{Laver2}.

\section{Preliminaries}

\subsection{Basics of BQO's}
We will first define what it means for a quasi order to be better. We will not use the classical definition developed by Nash-Williams \cite{Nash1}, but rather the topological one seen in \cite{ArgyrosStevo} and originally developed by Simpson \cite{Simpson}. This definition lends itself well to applications of the Galvin-Prikry theorem as seen in \cite{Carlos}. 
\begin{definition}
	Let $Q$ be a quasi order. We say $Q$ is a BQO if for any $f:[\NN]^\infty \rightarrow Q$, Borel with respect to the discrete topology of $Q$, there exists an infinite $A\subseteq \NN$ such that $\forall a\in [A]^\infty$ $f(a) \leqq f(a\setminus \text{min}a )$. 
\end{definition}
A nice property of BQO's is they behave a lot like well orders. In particular, we can do induction on them.
\begin{fact}(BQO induction)
	Let $Q$ be a BQO. To show a statement is true for all $q\in Q$, we may suppose it is true for the set $\{r\in Q: q\nleqq r \} $
\end{fact}
We will prove our desired results by coding our BQO's onto trees and utilizing the infinite tree theorem as Laver did.
\begin{definition}
	We denote $\mathcal{FT}_Q $ $\mathcal{T}_Q$ to be the class of all rooted trees that are finite or have height $\leq \omega$ respectively indexed by a quasi order $(Q,\leqq)$. We define the following quasi orders  $\leqq_1 $, $\leqq_I$ $\leqq_s$ and $\leqq_m$ on $\mathcal{T}_Q$.
	\begin{itemize}
		\item $(T,l) \leqq_1 (S,m) \iff$ there exists a injective $f:T\rightarrow S$ such that $ f(t_1\wedge t_2) = f(t_1)\wedge f(t_2)$ and $l(t) \leqq m(f(t)) $.
		\item $(T,l) \leqq_I (S,m) \iff$ $T=S$ and $l(t)\leqq m(t)$. 
		\item $(T,l) \leqq_s (S,m) \iff$ there exists a map $f:T\rightarrow S$ such that $t_1<_T t_2 \Rightarrow f(t_1)<_S f(t_2)$ such that $l(t) \leqq m(f(t)) $.
		\item $(T,l) \leqq_m (S,m) \iff$ there exists a map $f:T\rightarrow S$ such that $t_1\leq_T t_2 \Rightarrow f(t_1)\leq_S f(t_2)$ such that $l(t) \leqq m(f(t)) $.
	\end{itemize}   
\end{definition}
\begin{fact}(Infinite tree theorem)
	$Q$ BQO $\Rightarrow$ $ \mathcal{FT}_Q $ is BQO under $\leqq_1$ and $\leqq_m$, $\mathcal{T}_Q$ is BQO under $\leqq_1$ and $\leqq_s $. \cite{Nash1} \cite{Laver1}
\end{fact}
Elements in these sets will be denoted as pairs $(T,l)$, where $T$ is a rooted tree and $l:T\rightarrow Q$ is the labeling.
\begin{definition}
	Given a $q\in Q$, we define $1^q$ to be the one point tree indexed by $q$. 
\end{definition}
\begin{definition}
	Given $Q$ labeled trees $T_i$ indexed by some set $i\in I$, we define the tree $T=[q: T_i \; i\in I]$ to be the tree whose root is labeled by $q$ and branches into $T_i$ for each $i\in I$ i.e the immediate successors of the root, are the roots of $T_i$. In the case $I$ is linearly ordered, the lexicographical order of $T$ will be determined by $I$ in the natural way. 
\end{definition}
It may not be clear why the $\leqq_m$ relation would be useful, as $\leqq_1$ is a lot simpler and easier to understand. However, $\leqq_m$ will allow us to use Laver's covering theorem to take control of the number of treetops (sometimes referred to as leaves) our finite trees will have for the proof of theorem 1.2.
\begin{definition}
	Let $Q$ be a better quasi order. Take $W\subseteq \mathcal{FT}_Q $. We call $W^\prime \subseteq \mathcal{FT}_Q$ a cover of $W$ if for every $(T,l)\in W$ $\exists (S,m)\in W^\prime$ such that $S$ is a subtree of $T$ and $(T,l)\equiv_m (S,m) $. We say $W$ is $n$-coverable if there is a cover $W^\prime$ for which each $(S,m)\in W^\prime$ has at most $n$ treetops.
\end{definition} 
\begin{fact}(Laver's covering theorem)
	Let $Q$ be a better quasi order. For every $W\subseteq \mathcal{FT}_Q$, there is an $n$ for which $W$ is $n$-coverable. \cite{Laver2}
\end{fact}

\subsection{Basics of Linear Orders}
First, we will start with some notation. The letters L and M will be reserved for linear orders. 
\begin{definition}
	Given a linear order $L$, we define its order type $\text{tp}(L)$ to be the class of all linear orders isomorphic to $L$.
\end{definition}
Greek letters such as $\phi$ $\psi$ and $\varphi$ will be reserved for types. We will sometimes conflate linear orders with their types. We will try to do this as infrequently as possible, though we will ignore this rule entirely for special types like the rational type $\QQ$, regular cardinals $\kappa$ and the minimal Countryman types $C$, $C^*$. We now define the main quasi order on types that will be of interest to us. 
\begin{definition}
	Given two linear orders $L$ and $M$ of type $\phi$ and $\psi$ respectively, we say $\phi\leqq \psi$ if there exists an embedding of $L$ into $M$. If $\phi \leqq \psi $ and $\psi \leqq \phi$, we say $\phi \equiv \psi$. 
\end{definition}
The equivalence relation $\equiv$ is some times referred to as the biembeddability relation. There are many algebraic operations one can define on types. In particular, we can define sum and product
\begin{definition}
	Given $L$ and $M$ of type $ \phi$ and $\psi$ respectively, we define $\phi\cdot \psi$ as the order type of $L\times M$ with the antilexicographical ordering, 
\end{definition}
\begin{definition}
	Given $L$ and $M$ of type $\phi$ and $\psi$ respectively, we define $\phi+\psi$ to be the order type of $L\times \{0\}\cup M\times\{1\}$ with the antilexicographical ordering. 
\end{definition}
Interestingly, there is a way to describe linear ordered sums of types. This, for one gives us a lot more tools to algebraically analyze and construct types. Secondly, it also generalizes both finite sum and product. 
\begin{definition}
	Given a linear order $L$ and a collection of types $\{\phi_x:x\in L \}$, we define the type $\sum\limits_{x\in L} \phi_x$ to be $\text{tp}(\bigcup\limits_{x\in L} M_x\times\{x\}) $, where $\text{tp}(M_x)= \phi_x $ and $ \bigcup\limits_{x\in L} M_x\times\{x\}$ is ordered antilexicographically. 
\end{definition}
One can check that $\sum\limits_{x\in L} \phi \equiv \phi\cdot \text{tp}(L) $. Moreover, finite sums $\phi_1+..+\phi_k$ can be seen as the ordered sum $\phi_i$ over the linear order $\{1,...,k\}$ with the natural order. We also have a dual/reverse operation $*$ 
\begin{definition}
	Given a type $\phi$, the reverse $\phi^* $ is the type of $(L,>)$, where $\text{tp}(L,<) = \phi$. 
\end{definition}
Often to analyze an $L$-sum of types, it will be beneficial to break $L$ apart into disjoint convex pieces and work with the natural order they inherit from $L$. We will refer to this order as the block order.
\begin{definition}
	Given a linear order $(L,<)$ and a collection of disjoint intervals $\mathcal{I}$, the block order on $\mathcal{I} $ is defined by $A< B \iff \forall x\in A \forall y \in B$ $x<y$. 
\end{definition}
We will also need to work with labeled orders. They are defined near identically to labeled trees.
\begin{definition}
	Given a quasi order $(Q,\leqq)$ and a order type $\phi$, we call $\Phi = (L,l)$ a $Q$ labeled $\phi$ type if $\text{tp}(L) = \phi$ and $l:L\rightarrow Q$. Given two labeled types $\Phi = (L,l)$ and $\Psi=(M,m)$, we say $\Phi\leqq \Psi$ if there exists an embedding $f:L\rightarrow M $ such that $l(x) \leqq m(f(x)) $. For $q\in Q$, $1_q $ denotes the one pointed order labeled by $q$. 
\end{definition}

\subsection{On Ramsey Degrees}

Our interest in the final portion of this paper will be on decomposition properties held by Aronszajn lines. In particular, we will be interested in proving the existence of a Ramsey degree. 
\begin{definition}
	Given a linear order $L$, we write $L \rightarrow (L)_{k,n}^1 $ to mean that for $k\in \NN $ and any colouring $c: L\rightarrow k $, there exists $A\in [k]^n$ such that $c^{-1}(A)$ contains an isomorphic copy of $L$. If this holds for all $k$, we say $L$ has big Ramsey degree bounded by $n$. 
\end{definition}
\begin{fact}
	If $L \rightarrow (L)_{k,n}^1 $ and $\text{tp}(L) \equiv \text{tp}(M)$, then $M \rightarrow (M)_{k,n}^1 $.
\end{fact}
Consequently, rather than speak about a particular linear order, we can instead adopt the notation $\forall k \in \omega $ $\phi \rightarrow (\phi)_{k,n}^1$ to mean any linear order of type $\phi$ has big Ramsey degree bounded by n. 
\begin{lemma}
	Given order types $\phi, \psi$, if $\phi \rightarrow (\phi)_{k,n_1}^1 $ and $\psi \rightarrow (\psi)_{k,n_2}^1$, then $\phi+\psi \rightarrow (\phi+\psi)_{k,n_1+n_2}^1 $.
\end{lemma}
Many well known orders have Ramsey degree bounded by $1$. For example, the rationals $\QQ$, every regular cardinal and the generalized rationals $\eta_{\alpha \beta}$ \cite{Laver1}. A weaker type of Ramsey degree is the property of algebraic indecomposability.
\begin{definition}
	We call a type $\phi$ algebraically indecomposable (AI for short) if whenever $\phi \equiv \psi_1 + \psi_2$, there is an $i\in \{1,2\}$ for which $\phi \leqq \psi_i$
\end{definition}
In particular, being AI means having Ramsey degree bounded by 1 when we restrict our class of colourings to convex ones. For the class of $\sigma$-scattered orders, the subclass of AI types behave as the building blocks. We will show later that this remains true for Aronszajn types under PFA.  
\subsection{On Aronszajn Lines}
An Aronszajn line is any line of size $\aleph_1$ that is not isomorphic to a suborder of the reals $\RR$ and does not embed $\omega_1$ or $\omega_1^*$. There have been many constructions of Aronszajn lines. A special subclass of Aronszajn lines are Countryman lines. We call a lines $C$ Countryman if its square $C \times C$ (viewed as a product of posets, not linear orders) can be decomposed into countably many chains. Two classic examples of Countryman lines can be found in \cite{Saharon} and \cite{StevoWalks}. The latter of the two is also minimal in that for any Aronszajn line $A$ contains a copy of it or its reverse as a suborder. We will fix the name $C$ for this line.
\begin{fact}
	For any $L\in \{C,C^*,\QQ\}$, there exists a collection of disjoint intervals $\mathcal{I}$ which is isomorphic to $L$ under the block order. 
\end{fact}
The above fact will be vital for us. It is trivial to prove for $\QQ$, while for $C$ it requires a small forcing argument. In particular, it suffices to show that $C$ (and consequently $C^*$) contain no Souslin suborder. One can find the argument in \cite{Carlos} along with the following.
\begin{fact}{($\text{MA}_{\aleph_1} $)}
	For any $L\in \{C,C^*,\QQ\}$, $L^2 \equiv L $.
\end{fact}
Note, this is generally true for $\QQ$ and does not rely on Martins axiom. Another critical fact we will need is the existence of a universal Aronszajn line. Having one will mean that we can construct an $\omega_2$ sequence of AI Aronszajn lines cofinal under $\leqq$ that are constructible from $C$ under our algebraic operations. More on this will appear in the next section. 
\begin{fact}{(PFA)}
	Every Aronszajn line is either universal, or fragmented. \cite{Moore2}
\end{fact}
In order to use these two facts, for the rest of the paper, we shall be assuming PFA. 
\section{Laver's Tree Coding Argument}

\subsection{Shuffles and Constructing AI Types}
We first start by defining a shuffle. 
\begin{definition}
	For $L\in \{ C,C^*,\QQ \}$, we call the summation $\sum\limits_{z\in L} \phi_z $ an $L$-shuffle if 
	\begin{itemize}
		\item $\forall z \in L$ $\forall (u,v) \subseteq L$, $\exists z^\prime \in (u,v)$ such that $\phi_z \leqq \phi_{z^\prime}$
	\end{itemize}
\end{definition}
Shuffles are such that every interval of $L$ contains a cofinal collection of $\{\phi_x: x\in L\} $. For example, a simple product $L\times M$ is a shuffle. There is a similar notion for regular cardinals, which we will need. Note, shuffles extend naturally to quasi order labeled types $\Phi$ via the order $\leqq$ defined on them in section 2.2. 
\begin{definition}
	Given a regular cardinal $\kappa$, we call the type $\sum\limits_{\alpha<\kappa} \phi_\alpha$  $\kappa$ (resp. $\kappa^* $) unbounded if $\forall \alpha \exists \beta>\alpha $ such that $\phi_\alpha \leqq \phi_\beta$. 
\end{definition}
Note that shuffles are equivalent up to cofinality. That is, if $\forall \phi\in \{\phi_x:x\in L  \} $ $\exists \psi \in \{\psi_x :x\in L \}$ $\phi \leqq \psi$ and vice versa, then $\sum\limits_{x\in L} \psi_x \equiv \sum\limits_{x\in L} \phi_x$. The argument for why this is true is outlined in lemma 3.2 and is heavily reliant on fact 2.5.
\begin{definition}
	Given an $L$ sum of labeled $Q$ types $\Psi = \sum\limits_{x\in L} \Psi_x$, $\mathcal{U} = \{\Psi_x : x\in L  \}$, we say $\Psi$ is universal if $\Psi \geqq \Phi $ whenever $\Phi= \sum\limits_{x\in L} \Phi_x $ and $\mathcal{V}=\{\Phi_x: x\in L\}$ is dominated by $\mathcal{U} $ i.e $\forall \Gamma \in \mathcal{V}$ $\exists \Theta \in \mathcal{U}$ $\Gamma \leqq \Theta$. 
\end{definition}
\begin{fact}
	Every $L$ shuffle for $L\in \{\QQ,C,C^*\}$, $L$ unbounded sum $L\in \{\omega,\omega^*\}$ of $\mathcal{U}$ is universal. 
\end{fact}
Consequently, universal types over $L\in \{\omega,\omega^*,\QQ,C,C^*\}$ are simply $L$-shuffles or $L$-unbounded sums respectively up to equivalence. 
\\
\\
One nice property of shuffles is that every fragmented Aronszajn line can be embedded into a recursively constructed shuffle. Martinez Ranero originally constructed these lines, though the notion of shuffle was not considered.  
\begin{fact}{($\text{MA}_{\aleph_1}$)}
	If $\phi$ is a fragmented Aronszajn type, there exists an $\alpha<\omega_2 $ such that $\phi\leqq \text{tp}(D_\alpha^{+})$, where $D_\alpha^+$ is recursively defined as $C$-shuffles like so
	\begin{itemize}
		\item $D_0^+ =C$
		\item $D_0^- = C^*$
		\item For $\alpha = \beta+1$ $D_\alpha^+= C\times D_\beta^-$
		\item For $\alpha = \beta+1$ $D_\alpha^-= C^*\times D_\beta^+$
		\item For $\alpha$ limit $D_\alpha^+$ is a $C$ shuffle of $\{D_\beta^+: \beta<\alpha \}$
		\item For $\alpha$ limit $D_\alpha^-$ is a $C^*$ shuffle of $\{D_\beta^-: \beta<\alpha \}$
	\end{itemize}
\end{fact}
\begin{fact}
	Every fragmented Aronszajn line $\phi$ of rank $\alpha $ embeds into either $D_\alpha^+$ or $D_\alpha^- $
\end{fact}
A useful property of shuffles and unbounded sums is that they preserve algebraic indecomposability. 
\begin{lemma}
	For $L \in \{C,C^*,\QQ \}$, if $\phi$ is an $L$ shuffle of algebraically indecomposable elements, then $\phi$ is algebraically indecomposable. 
\end{lemma}
\begin{proof}
	Suppose $\bigcup\limits_{x\in L} M_x$ is an $L$ shuffle with type $\phi$. Suppose $\bigcup\limits_{x\in L} M_x\subseteq A\cup B $ where $A<B$. We may suppose both $A$ and $B$ are nonempty. Then, there is a $z\in L$ for which $\bigcup\limits_{x<z} M_z \subseteq A$. However, as $\phi$ is a shuffle, $\text{tp}(\bigcup\limits_{x<z} M_z) \equiv \phi$.        
\end{proof}
Note however, they do not preserve Ramsey degrees. For example, consider the unbounded sum $\sum\limits_{n\in \omega}\phi_i$ such that $\phi_i \equiv C$ for $i$ even else $\phi_i\equiv C^*$. One can check that this Aronszajn line has Ramsey degree $2$. This can be mitigated if the shuffle is strict. 
\begin{definition}
	We call an $L$-shuffle $\sum\limits_{x\in L} \phi_x $ strict $\iff$ $\{\phi_x: x\in L\}$ forms a $\leqq$-chain.  
\end{definition}
\begin{lemma}
	Suppose $\phi = \sum\limits_{x\in L} \psi_x$ is a strict $L$ shuffle. Then $\phi \equiv \sum\limits_{x\in L} \varphi_x $ where $\{\varphi_x:x\in L \}$ is well ordered under $\leqq$ and has order type $1$ or $|L| $    
\end{lemma}
\begin{proof}
	We can take $\{\varphi_y: y\in \kappa \} \subseteq \{\psi_x:x\in L \} $ cofinal with $\alpha <\beta \Rightarrow \varphi_\alpha<\varphi_\beta$. There are two cases to consider.\\
	\\
	\textbf{Case 1:} $\kappa = |L|$. Consider the shuffle $\sum\limits_{y\in L} \varphi_y$ where each member in the cofinal sequence appears exactly once. It is clear this is a shuffle as every interval has size $|L|$ and hence $\{\varphi_y: y\in I \} $ is cofinal for any interval $I$. Consider $\mathcal{I}$ a collection of disjoint intervals in $L$ isomorphic to $L$ under the block sequence order. Let $f$ be such an isomorphism. It is clear now by the cofinality that $ \sum\limits_{x\in L} \psi_x \leqq \sum\limits_{y\in L} \varphi_y$. For each $x\in L$, we take $z\in f(x)$ such that $\psi_x \leqq \varphi_z $. Doing the same cofinality trick, we can reverse the argument to get $\sum\limits_{y\in L} \varphi_y \equiv \sum\limits_{x\in L} \psi_x $.\\
	\\
	\textbf{Case 2:} $\kappa< |L|$. Consequently, $\kappa$ can embed into $L$. Consider now $\varphi = \sum\limits_{\alpha<\kappa}\varphi_\alpha  $. Take an interval partition $\mathcal{I}$ of $L$ isomorphic to $L$. Within each interval, take $\kappa$ increasing new disjoint intervals and for the $\alpha$ interval, find a type $\psi_\alpha \geqq \varphi_\alpha $. So, for each $I\in \mathcal{I}$, we found a type of the form $\sum\limits_{\alpha<\kappa}\varphi_\alpha $. Note $\sum\limits_{\alpha<\kappa} \psi_\alpha \geqq \sum\limits_{\alpha<\kappa}\varphi_\alpha$. However, by cofinality of $\varphi_\alpha$, $\sum\limits_{\alpha<\kappa} \psi_\alpha \equiv \sum\limits_{\alpha<\kappa}\varphi_\alpha$. We have thus shown that $\text{tp}(L)\times \varphi \leqq  \sum\limits_{x\in L} \psi_x$. However, for all $x\in L$, $\psi_x \leqq \varphi$, so  $\text{tp}(L)\times \varphi \equiv  \sum\limits_{x\in L} \psi_x$. 
\end{proof}
\begin{lemma}
	Suppose $\phi= \sum\limits_{x\in L} \psi_x$ is a strict $L$ shuffle with $L\in \{C,C^*,\QQ\}$ and there is an $n\in \omega$ such that $\forall x\in L $, $\psi_x \rightarrow (\psi_x)_{k,n}^1$, then $\phi \rightarrow (\phi)_{k,n}^1 $      
\end{lemma}
\begin{proof}
	Let $\bigcup\limits_{x\in L} M_x$ be such that $\text{tp}(M_x)=\psi_x $ and is a $L$ shuffle. By the previous lemma, we may suppose $\{\psi_x:x\in L \} $ is well ordered under $\leqq$. Moreover, by the previous lemma, there are two cases to consider. Either $\{\psi_x :x\in L\} $ has order type $|L|$ or $1$. \\
	\\
	\textbf{Case 1:} $ \phi \equiv \text{tp}(L)\times \varphi$ where $\varphi$ is a $|L|$ unbounded sum of types from $\{\psi_x: x\in L\}$ i.e the cofinality $1$ case. It is clear that $\varphi \rightarrow (\varphi)_{k,n}^1 $. Then $\phi\rightarrow (\phi)_{k,n}^1  $.\\
	\\
	\textbf{Case 2:} $ \forall x,y\in L$ $\psi_x<\psi_y$ or $\psi_x>\psi_y$. Let $c:  \bigcup\limits_{x\in L} M_x \rightarrow k$ with $k\geq n $. For each $x\in L$, there is an $A\subseteq [n]^k $ such that $\exists M_x^\prime \subseteq M_x \cap c^{-1}(A)$. Consider $c:L\rightarrow [n]^k$ to map each $x$ to some $A\in [n]^k$ such that we can find an $M_x^\prime $ as above. Since $L$ is indecomposable, we can find $L^\prime\subseteq L$ isomorphic to $L$ and such that $c\upharpoonright L^\prime = A $. For each $x\in L^\prime$, we find the requisite $M_x^\prime$. It is clear that $\bigcup\limits_{x\in L^\prime} M_x^\prime \subseteq c^{-1}(A)$. However, as $\psi_x$ is $\leqq$ well ordered of order type $|L|$ and each $\text{tp}(M_x^\prime)$ are distinct, $\bigcup\limits_{x\in L^\prime} M_x^\prime $ is an $L$ shuffle. Since $\text{tp}(M_x^\prime)$ is cofinal in $\{\psi_x:x\in L \}$, $\text{tp}(\bigcup\limits_{x\in L^\prime} M_x^\prime) \equiv \phi$ as desired. Hence, $\phi \rightarrow (\phi)_{k,n}^1$. 
\end{proof}
We now define our class of interest. We cannot isolate just Aronszajn lines with our method as we want our class to be hereditarily closed. Consequently, we will be interested in the class of Fragmented Aronszajn lines and countable orders. Note, every countable order can be embedded into $\QQ$. 
\begin{definition}
	We will define $\mathcal{C}$ to be the class of all Fragmented Aronszajn lines and countable orders. We also recursively construct $\mathcal{C}_\alpha$ for $\alpha <\omega_2$ as follows.
	\begin{itemize}
		\item $\mathcal{C}_0 = \{0,1\}$.
		\item $\phi \in \mathcal{C}_\beta \iff \phi = \sum\limits_{x\in L} \phi_x$, $\phi_x \in \bigcup\limits_{\alpha<\beta} \mathcal{C}_\alpha$ and $L\in \{C,C^*,\QQ\} $
	\end{itemize}
\end{definition}
It is clear that $D_\alpha^+, D_\alpha^- \in \mathcal{C}_\alpha$, It is also the case that for any Aronszajn $\phi \in \mathcal{C}_\alpha$, either $\phi \leqq D_\alpha^+$ or $\phi \leqq D_\alpha^-$. One can prove this by induction. Suppose it is true for all $\mathcal{C}_\beta$ for $\beta<\alpha $. Take $\phi \in \mathcal{C}_\alpha$. $\phi = \sum\limits_{x\in L} \phi_x $, $L\in \{C,C^*,\QQ\}$, where $\forall x\in L$ $\exists \beta<\alpha$ $\phi_x \leqq D_\alpha^+$ or $D_\beta^-$. So, $\phi \leqq  \sum\limits_{x\in L} \psi_x$ where $\psi_x\equiv D_\beta^+$ or $D_\beta^-$ for some $\beta <\alpha$. Note that one of $\{\psi_x :x\in L,\;\exists \beta<\alpha \;  \psi_x \equiv D_\beta^+ \} $ $\{\psi_x :x\in L,\; \exists \beta<\alpha \;  \psi_x \equiv D_\beta^- \} $ is cofinal in $\{\psi_x: x\in L \}$. Consequently, an $L$ shuffle of one of them can embed $\phi$ (as shuffles are universal and equivalent up to cofinality). If $L=C$ or $C^*$, we are done. If $L=\QQ$, then the rank of $\phi$ as an aronszajn line is $\text{sup} \{\beta: \exists x \in L \;  \psi_x \equiv D_\beta^+ \text{ or } D_\beta^- \}\leq \alpha$ and so $\phi \leqq D_\alpha^+$ or $D_\alpha^-$. 
\begin{lemma}
	$\mathcal{C} = \bigcup\limits_{\alpha <\omega_2} \mathcal{C}_\alpha$.
\end{lemma}
\begin{proof}
	The inclusion $\mathcal{C}_\alpha \subseteq \mathcal{C}$ for all $\alpha<\omega_2$ is clear. It suffices to show that and $\phi \in \mathcal{C} \Rightarrow \phi \in \mathcal{C}_\alpha$ for some $\alpha <\omega_2$.\\
	\\
	\textbf{Claim:} $\forall \phi \in \bigcup\limits_{\alpha <\omega_2} \mathcal{C}_\alpha$, $\sum\limits_{x\in \phi} \psi_x \in \bigcup\limits_{\alpha <\omega_2} \mathcal{C}_\alpha $ where $\psi_x\in \bigcup\limits_{\alpha <\omega_2} \mathcal{C}_\alpha $. 
	\begin{proof}
		We will show this by way of induction on $\alpha$ where we prove for any $\alpha$, $\phi \in \mathcal{C}_\alpha$ and $\forall x\in \phi, \psi_x \in \bigcup\limits_{\gamma < \alpha}\mathcal{C}_\gamma $ $\sum\limits_{x\in \phi}\psi_x \in \bigcup\limits_{\alpha <\omega_2} \mathcal{C}_\alpha $. It is clearly true for $\alpha =0$. Suppose it is true for all $\alpha<\beta  $. Take $\phi \in \mathcal{C}_\beta$ and $\psi_x \in \mathcal{C}_\beta $. Let
		\begin{align*}
		\Psi &= \sum\limits_{x\in \phi} \psi_x
		\end{align*}
		Without loss of generality, suppose $\phi \leqq D_\beta^{+}$, hence we may assume $\Psi = \sum\limits_{x\in D_\beta^{+}} \psi_x $. But of course, $D_\beta^{+} = \sum\limits_{y\in C} \varphi_y$, where $\varphi_y \in \bigcup\limits_{\alpha<\beta} \mathcal{C}_\alpha $. By our induction hypothesis, for each $y\in C$, $\sum\limits_{x\in \varphi_y} \psi_y \in \bigcup\limits_{\alpha<\omega_2} \mathcal{C}_\alpha$. Consequently,
		\begin{align*}
		\Psi &= \sum\limits_{y\in C} \sum\limits_{x\in \varphi_y} \psi_x \in \bigcup\limits_{\alpha<\omega_2} \mathcal{C}_\alpha
		\end{align*}
		To generalize to arbitrary sums of the form $\phi \in \bigcup\limits_{\alpha <\omega_2} \mathcal{C}_\alpha $, $\forall x\in \phi$ $\psi_x\in \bigcup\limits_{\alpha <\omega_2} \mathcal{C}_\alpha $, one simply needs to find $\beta$ large enough so that all $\phi $ and $\psi_x $ embed into  $D_\beta^{+}$. 
	\end{proof}
	Since $\bigcup\limits_{\alpha <\omega_2} \mathcal{C}_\alpha$ is closed under sums, for any $\phi \in  \mathcal{C}$, we can take $\alpha$ large enough so that we can find $L \subseteq D_\alpha^{+}$ with $\text{tp}(L) =\phi$. Consequently $\sum\limits_{x\in D_\alpha^{+}} \mathds{1}_x = \phi \in \bigcup\limits_{\alpha <\omega_2} \mathcal{C}_\alpha$, wher $\mathds{1}_x =1  \iff x\in L$ else $\mathds{1}_x =0$. 
\end{proof}
We now construct a class of $Q$ labeled lines that will turn out to be the building blocks of $\mathcal{C}(Q)$, where $\mathcal{C}(Q)$ are the $Q$ labeled types from $\mathcal{C}$. The reason for wanting to work with labeled lines is that we can iteratively work with order types indexed by order types. 
\begin{definition}
	Given a BQO $Q$, we define $\mathcal{H}(Q)$ recursively. $\mathcal{H}_0(Q) = \{0,1_q\}$. $\phi \in \mathcal{H}_\beta(Q) \iff \phi$ is an $L$ shuffle of members $\mathcal{U} \subseteq \bigcup\limits_{\alpha<\beta} \mathcal{H}_\alpha(Q)$, $L\in \{\QQ,C,C^* \}$ or is an $L$ unbounded sum for $L\in \{\omega,\omega^*\} $. 
\end{definition}
In the case $Q= \{0,1\} $, we simply write $\mathcal{H}$ as the class is identifiable with the class of orders with no labels.
\begin{prop}
	Every $\phi \in \mathcal{H}$ is algebraically indecomposable.
\end{prop}
\begin{proof}
	We do this by induction. Take $\phi \in \mathcal{H}_\beta$ and suppose the statement is true for all $\alpha<\beta$. By lemma 1.2, the case in which $\phi = \sum\limits_{x\in L} \phi_x$, $\forall x\in L \exists \alpha<\beta $ $\phi_x\in \mathcal{H}_\alpha$ is an $L$-shuffle has been accounted for. Suppose instead $\phi= \sum\limits_{k\in \omega} \phi_k $ an unbounded sum. The case with $\omega^*$ is symmetric. Suppose $ L = \bigcup\limits_{k\in \omega }M_k$ is a realization of $\phi$. Let $L=A\cup B$ where $\forall x\in A $ $y\in B$ $x<y$ and neither is empty. Then there is a $m\in \omega $ for which $\bigcup\limits_{k\in \omega\setminus m} M_k\subseteq B $. However, it is clear that $\sum\limits_{k\in \omega} \phi_k \equiv \sum\limits_{k\in \omega\setminus m} \phi_k$. 
\end{proof}
\subsection{Aronszajn Lines are Finite Sums of AI Types}
In this section, we show every labeled type in $\mathcal{C}(Q)$ can be expressed as a finite sum of AI labeled types from $\mathcal{H}(Q)$. But first, we must show that $\mathcal{H}(Q)$ is a BQO. 
\begin{definition}
	Given a BQO $Q$, we define $Q^+$ to be the disjoint union $Q\cup \mathcal{C} $. We also define $\mathcal{T}_{Q}$ to be the BQO of trees indexed by $Q$ under the $\leqq_m$ ordering. 
\end{definition}
Consider the map $T:\mathcal{H}(Q) \rightarrow  \mathcal{T}_{Q^+}$ constructed recursively as follows. $T(0)=0$, $T(1_q)= 1^q $ If $\Psi \in \mathcal{H}_\beta(Q) \setminus \bigcup\limits_{\alpha<\beta}\mathcal{H}_\alpha(Q)$ and is a L-shuffle or $L$ unbounded sum of some types $\mathcal{U}\subseteq \mathcal{H}(Q) $, define $T(\Psi)= [L: \{T(\Theta):  \Theta \in \mathcal{U}\} ] $.
\begin{prop}
	$T(\Phi) \leqq_{s} T(\Psi) \Rightarrow \Phi\leqq \Psi$.
\end{prop}
We do this by induction on $\alpha$. Suppose for all $\alpha<\beta $, $\Psi \in \mathcal{H}_\alpha \Rightarrow$ $T(\Phi)\leqq_s T(\Psi) \Rightarrow \Phi \leqq \Psi$. It is clear that this is true for $\mathcal{H}_0(Q)$. Take $\Psi \in \mathcal{H}_\beta (Q)$. Suppose $T(\Phi)\leqq_s T(\Psi)$ and is witness by $f: T(\Phi) \rightarrow T(\Psi)$. 

\textbf{Case 1:} $f(\text{root}(T(\Phi))) \neq \text{root}(T(\Psi))$. Since $T(\Psi) = [L: T(\Theta) \; \Theta \in \mathcal{U}] $ for some $L$ and some $\mathcal{U} \subseteq \bigcup\limits_{\alpha<\beta} \mathcal{H}_\alpha(Q)$,  $\exists \Theta \in \mathcal{H}_\alpha(Q)\cap \mathcal{U}$, $\alpha<\beta$ for which $T(\Phi) \leqq_s T(\Theta) $. By our induction hypothesis, $\Phi \leqq \Theta$. However, $\Psi$ is an $L$ shuffle or unbounded sum of $\mathcal{U}$. In particular, $\Theta \leqq \Psi$ and we are done by transitivity.\\
\\
\textbf{Case 2:} $f(\text{root}(T(\Phi)) ) = \text{root}(T(\Psi) ) $. So, $T(\Psi) = [M : T(\Theta) \; \Theta \in \mathcal{U}]$ and $T(\Phi) = [L: T(\Gamma) \Gamma \in \mathcal{V} ] $. It is clear from $f$ that for each $\Gamma \in \mathcal{V}$ $\exists \Theta \in \mathcal{U}$ such that $T(\Gamma)\leqq_s T(\Theta) \Rightarrow \Gamma \leqq \Theta$ by our induction hypothesis. \\
\\
If $M= \omega$ (resp. $\omega^* $), then $L=\omega$ and $\Phi = \sum\limits_{k\in \omega} \Gamma_k $ $\Psi = \sum\limits_{k\in \omega} \Theta_x$. For each $\Gamma_k \exists \Theta_{m_k}$ such that $\Gamma_k\leqq \Theta_{m_k} $. Since $\Phi$ is an unbounded sum, we can take $m_k$ strictly increasing to build an embedding $\Phi \leqq \Psi$. 
\\
\\
Suppose instead that $M \in \{\QQ,C,C^*\}$. It follows that $L\leqq M$. Consequently, there is an interval partition $\mathcal{I}$ of $M$ that is isomorphic to $L$. Let $\iota:L\rightarrow \mathcal{I}$ be an isomorphism. Since $\Psi= \sum\limits_{x\in M} \Theta_x$ and $\Phi = \sum\limits_{x\in L}\Gamma_x$ and $\Psi$ is a shuffle, for each $x\in L \exists z \in \iota(x)$ such that $\Gamma_x \leqq \Theta_z$. Consequently, we can recursively construct an embedding $\Phi \leqq \Psi$ as desired. 
\\
\\
From the above, it follows that $\mathcal{H}(Q)$ is BQO for any BQO $Q$. 
\begin{lemma}
	For any BQO $Q$, a $Q$ labeled $L \in \{\QQ,C,C^*\}$ can be expressed as a countable sum of universal $L$ types.  
\end{lemma}
\begin{proof}
	We may suppose that for all $r\in Q$, the statement is true for the BQO $ \{q\in Q: r\nleqq q \}$. Let $(L,l)$ be a $Q$ labeling. Consider the equivalence relation $\sim$ on $L$,$x\sim x$, $x\sim y \iff \forall (u,v) \subseteq (x,y)$, $(u,v)$ is a countable sum of universal $L$ types. Each class is a countable sum of universal $L$ types as $(x,y) \equiv L$ and has countable cofinality. So, suppose otherwise. Take $X,Y \in L/\sim $. It is clear that $(X,Y) \equiv L$. Moreover, for all $q\in Q$, $\exists Z\in (X,Y)$ and $z\in Z$ such that $l(z) \geqq q$. If not, then any $z\in \cup_{Z\in(X,Y)} Z $ satisfies $q \nleqq l(z)$ and we are in our base case. Hence, $X=Y$ and we have a contradiction. It follows then that our original labeled line $(L,l)$ was universal.    
\end{proof}
\begin{lemma}
	Given $\chi \in \mathcal{H}(\mathcal{H} (Q))$, $\chi = (X, l)$, the $Q$ labeled order order $\overline{\chi} = \sum\limits_{x\in X} l(x)$ is in $\mathcal{H}(Q) $
\end{lemma}
\begin{proof}
	This can be done via a simple induction argument. It is clear that statement is true for $\mathcal{H}_0 (\mathcal{H}(Q))$. If it is true of $\mathcal{H}_\alpha (\mathcal{H}(Q))$ for every $\alpha<\beta$, then converting $M$ to shuffle or unbounded sum of members from $\mathcal{H}_\alpha (\mathcal{H}(Q)) $, we can imply our induction hypothesis and conclude as $\mathcal{H}(Q)$ is closed under shuffles and unbounded sums. 
\end{proof}
\begin{theorem}
	Let $Q$ be a BQO. Any $Q$ type $\Phi \in \mathcal{C}_\alpha(Q)$ can be written as a finite sum of members from $\mathcal{H}(Q)$. 
\end{theorem}
\begin{proof}
	We do this by induction on $\beta$. It is definitely true for $\beta=0$. Suppose the statement is true for all $\alpha<\beta $. If the base of $\Phi$ is countable, we are done by Laver's theorem \cite{Laver1} so suppose otherwise. By lemma 3.5, $\Phi = \sum\limits_{x\in L} \Phi_x$, $\text{bs}(\Phi_x) \in \bigcup\limits_{\alpha <\beta }\mathcal{C}_\alpha $, $L\in \{\QQ,\omega,\omega^*,C,C^* \}$. \\
	\\
	\textbf{Case 1:} $L\in \{\omega , \omega^* \} $. Then by our induction hypothesis, $\Phi = \sum\limits_{k\in L} (\Psi_{k,1} +...+ \Psi_{k,r_k}) $ where $\Psi_{k,r_k} \in \mathcal{H}(Q)$ and $r_k$ is a sequence of integers. But then, up to reorganization, $\Phi = \sum\limits_{k\in \omega} \Theta_k$, where $\Theta_k \in \mathcal{H}(Q) $. Since $\mathcal{H}(Q)$ is BQO, there is a minimal $m$ for which $\forall k>m$, $\Theta_m \nleqq \Theta_k$. If not, we can construct a nowhere increasing subsequence $\Theta_{m_k} $, contradicting the BQO assumption. But then, $\sum\limits_{k\in\omega \setminus \{0,...,m\} } \Theta_k$ is an unbounded sum of members from $\mathcal{H}(Q) $ and hence, a member of $\mathcal{H}(Q) $ itself. But then, $\Phi = \Theta_1+...+\Theta_m + \sum\limits_{k\in\omega \setminus \{0,...,m\} } $, a finite sum of members from $\mathcal{H}(Q)$
	\\
	\\
	 \textbf{Case 2:} $L \in \{\QQ,C,C^*\} $. By our induction hypothesis, $\Phi = \sum\limits_{x\in L} (\Phi_{x,1}+..,\Phi_{k_x})$ where $k_x$ is an integer. As $L^2 \equiv L$, up to reorganization, we may suppose that $\Phi =  \sum\limits_{x\in L} \Phi_x$ where each $\Phi_x $ is in $ \mathcal{H}_\alpha(Q)$ for some $\alpha<\beta$. Consider the $\chi \in \mathcal{H}_1(\mathcal{H}(Q) )$, $\chi = (L,l)$, $l(x)= \Phi_x$. By lemma 3.5, $\chi = \sum\limits_{x\in S} \Psi_x^\prime$, $\Psi_x^\prime \in \mathcal{H}(Q)$ and $S$ countable. By Laver's theorem, $\text{tp}(S) = \Omega_1+...+\Omega_k$ where $\Omega_i \in \mathcal{H}(\{0,1\}) $. So then, $\chi = \chi_1+...+\chi_k$ where each $\chi_i \in \mathcal{H}( \mathcal{H}(Q))$. But then, $\Phi \equiv \overline{\chi} = \overline{\chi_1}+...+\overline{\chi_k}$ and we are done by lemma 3.6.  
\end{proof}
\begin{corollary}(PFA)
	Given a BQO $Q$, the class $ $$\mathcal{C}(Q)$ is BQO.
\end{corollary}
\begin{corollary}(PFA)
	The class of all Aronszajn lines is BQO.
\end{corollary}
\begin{corollary}
	In conjunction with Lavers' result from \cite{Laver1}, the class of all Aronszajn lines and $\sigma$-scattered orders closed under summations over one another is BQO. 
\end{corollary}
\subsection{Aronszajn Lines have Finite Ramsey Degree}
Being able to decompose Aronszajn lines into a finite sum of AI types is a quite powerful result. In particular, it means that our shuffles could have been strict with no change to the class $\mathcal{H}$. In this subsection, we will use our newly found results to show that every member in $\mathcal{H} $ could have been coded by a finite tree labeled with AI types. First, we must define the class of trees. 
\begin{definition}
	Consider the class $\mathcal{U} \subseteq \mathcal{FT}_{\mathcal{H}}$ defined recursively as follows. $\mathcal{U}_0 =\{\emptyset, 1_1\}$ where $1_1$ is the one node tree indexed by $1$ and $(T,l) \in \mathcal{U}_\alpha$ if and only if one of the following holds
	\begin{itemize}
		\item $(T,l)\in \mathcal{U}_\beta$ for some $\beta<\alpha $
		\item  $\exists L\in \{\QQ,C,C^*\}$, if $\forall x\in L$, $(T,l_x) \in \mathcal{U}_\beta$ for some $\beta<\alpha $ and $\forall x,y,z\in L $, $\exists u\in (x,y)$ such that $(T,l_z) \leqq_I (T,l_x) $, then $\sum\limits_{x\in L} (T,l_x) = (T,l) \in \mathcal{U}$.
		\item $(T,l) = \sum\limits_{n\in \omega}(T,l_n $ and $\forall n\in \omega $ $\exists \beta<\alpha $ $(T,l_n) \in \mathcal{U}_\beta$ and $n<m \Rightarrow (T,l_n) \leqq_I (T,l_m)$.
		\item $(T,l)=[\omega; (T_1,l_1),...,(T_n,l_n)]$ (resp. $[\omega^*; (T_1,l_1),...,(T_n,l_n)] $) with $(T_1,l_1),...,(T_n,l_n) \in \mathcal{U}_\beta$ for some $\beta<\alpha $ and $(T_i,l_i) = \sum\limits_{n\in \omega} (T_i,l_{i,n})$ (resp. $\sum\limits_{n\in \omega} (T_i,l_{i,n})$) and for $i<k$ $(T_i,l_{i,n}) \leq_I (T_k,l_{k,n})$. 
	\end{itemize}
\end{definition}
\begin{definition}
	Given $(T,l) \in \mathcal{U}$, we assign a linear order $\overline{(T,l)}$ in $\mathcal{H}$ recursively as follows.
	\begin{itemize}
		\item $\overline{1_1}=1$
		\item $ \overline{\emptyset} =0$
		\item If $(T,l) = \sum\limits_{x\in L} (T,l_x)$ for $ L \in \{\QQ,C,C^*,\omega,\omega^* \}$, then $\overline{(T,l)} = \sum\limits_{x\in L} \overline{(T,l_x)}$
		\item If $(T,l) = [L; (T_1,l_1),...,(T_n,l_n)] $  for $L\in \{\omega,\omega^* \}$, $\overline{(T,l)} = \sum\limits_{x\in L } \overline{(T_1,l_{1,x} )}+...+\overline{(T_n,l_{n,x} )} $
	\end{itemize}
	We let $\mathcal{U} = \bigcup\limits_{\alpha \in \text{Ord}} \mathcal{U}_\alpha $
\end{definition}
From the set up, it should be clear that our proof is going to require an induction proof. The following lemma will allow us to easily compare $L$ and $M$ shuffles to one another when $L$ and $M$ are independent with respect to $\leqq$. 
\begin{lemma}
	Suppose $L_1, L_2 \in \{\omega,\omega^*,C,C^*, \QQ \}$ and $\sum\limits_{x\in L_1}\phi_x \leqq \sum\limits_{x\in L_2} \psi_x $ where both sums are shuffles of AI objects. Then, at least one of the following must occur.
	\begin{itemize}
		\item $\exists z\in L_2$ for which $\sum\limits_{x\in L_1}\phi_x \leqq \psi_z $\\
		\item $L_1\leqq L_2 $
	\end{itemize}
\end{lemma}
\begin{proof}
	Take an embedding $f: \sum\limits_{x\in L_1}M_x \leqq \sum\limits_{x\in L_2} N_x $. The result is trivial if $L_1,L_2\in \{\omega,\omega^*\}$ so we ignore these cases.  Note that for a given $M_x$, $f[M_x] $ cannot be cofinal. Consider the mapping $g$ from $L_1$ into bounded sets of $L_2$ that maps $x$ to the set $\{y\in L_2 : P_y\cap f[M_x] \neq \emptyset \}  $. Note that $g$ has the property that $x<_{L_1} y \Rightarrow g(x) < g(y) $ in the block sequence order or $g(x)\cap g(y)$ is a singleton and $\text{max}g(x) = \text{min} g(y) $. Consider the equivalence relation $\sim$ defined as follows. 
	\begin{align*}
	\forall x,y\in L_1 \forall u,v\in (x,y), u<v, \; \text{max}g(u)=\text{min} g(v)  \Rightarrow x\sim y\\
	\forall x\in L_1 x\sim x\\
	x\sim y \Rightarrow y\sim x
	\end{align*}
	It is clear that $\sim$ is an equivalence relation with convex equivalence classes. There are two cases to consider. The first is that $\exists X\in L_1/\sim $ such that $X\equiv L_1$. In this case, $\forall u\in X$, $g(u)$ is a singleton and $\sum\limits_{x\in X} M_x \equiv \sum\limits_{x\in L_1} M_x$ on the account that $X$ is convex and the sum was a shuffle. But then, $f\upharpoonright{\sum\limits_{x\in X} M_x} $ has range in $N_{g(u)}$ for some $u\in X$ and we are done.\\
	\\
	Suppose instead that no $X\in L_1/\sim$ is equivalent to $L_1$. In this case, $\{g(x) : x\in L_1 \}$ forms a block sequence. Taking a selector $s:L_1 \rightarrow L_2$, $s(x) \in g(x)$, we have shown $L_1\leqq L_2$. 
\end{proof}
In our main proof, we will see how this extends to trees under the $\leqq_m $ ordering. Given that the trees our finite, we will get a similar result by applying the above lemma a finite number of times.\\
\\
Akin to proposition 2, we must hope that two trees being comparable with respect to $\leqq_m$ gives us some tangible information about the orders they code. Fortunately, this is true. This requires an exhaustive case analysis. 
\begin{prop}
	If $(S,l) \leqq_m (T,m)$, then $\overline{(S,l)} \leqq \overline{(T,m)}$.
\end{prop}
\begin{proof}
	We work on induction on $\mathcal{U}_\alpha$. Suppose for all $ (T^\prime,m^\prime) \in \mathcal{U}_\alpha$ for $\alpha<\beta$, $(S,l)\leqq_m (T^\prime,m^\prime) \Rightarrow \overline{(S,l)}\leqq \overline{(T^\prime,m^\prime)} $. Take $(T,m) \in \mathcal{U}_\beta $. Suppose $(S,l) \leqq_m (T,m) \Rightarrow \overline{(S,l)} \leqq \overline{(T,m)} $ for $(S,l) \in \mathcal{U}_\alpha$ for all $\alpha<\gamma$. Take $(S,l) \leqq_m (T,m)$ with $(S,l)\in \mathcal{U}_\gamma$. We may also assume the result is true for all $(T^\prime,m^\prime) <_1 (T,m) $ and $(S^\prime,l^\prime) <_m (S,l)$ by BQO induction. Note, the statement is trivial if either $S$ or $T$ is a singleton.\\
	\\
	\textbf{Case 1a:} $(T,m) = [\omega; (T_1,m_1),...,(T_n,...,m_n)] $ and $(S,l) = [\omega; (S_1,l_1),...,(S_k,l_k)]$. Then, $\forall i \exists j $ such that $(S_i,l_i) \leqq_m (T_j,m_j) $. Suppose it is the case that $\overline{(S_i,l_i)}$ and $ \overline{(S_k,l_k)}$ both embed into $\overline{(T_j,m_j)}$. As $\overline{(T_j,m_j)} = \sum\limits_{x\in \omega}\overline{(T_j,m_{j,x})} \equiv \sum\limits_{x\in \omega \text{even}}\overline{(T_j,m_{j,x})} \equiv \sum\limits_{x\in \omega \text{ odd}}\overline{(T_j,m_{j,x})}  $, we can embed both $(S_i,l_i)$ and $(S_k,l_k)$ simultaneously. Consequently, one can further embed $\overline{(S,l)}$ into $\overline{(T,m)}$. The case of $(S,l) = [\omega^*; (S_1,l_1),...,(S_k,l_k)] $ is symmetric. \\
	\\
	\textbf{Case 1b:} $(T,m) = [\omega; (T_1,m_1),...,(T_n,...,m_n)] $ and $(S,l) = \sum\limits_{x\in L} (S,l_x)$ for some $L\in \{\QQ,C,C^*,\omega,\omega^*\}$. Then it must follow that $(S,l) \leq_m (T_i,m_i)$ for some $i$ and we are done. \\
	\\
	\textbf{Case 2a:} $(T,m) = \sum\limits_{x\in M} (T,m_x)$ for $M\in \{\omega,\omega^*,C,C^*,\QQ\}$ and for each $x\in M$, $(T,m_x) \in \mathcal{U}_\alpha$ for some $\alpha<\beta $. $(S,l) =  \sum\limits_{x\in L} (S,l_x)$ for some $L\in \{C,C^*,\omega,\omega^*,\QQ\}$. Let $f:S\rightarrow T$ be an embedding witnessing $(S,l) \leqq_m (T,m)$. If $L\nleqq M$, then by lemma 2.4, for each $s\in S$, $\exists z_s \in M$ such that $f(l(s))\leqq m_{z_s}(f(s))$. Since $\{(T,m_z) z\in M\} $ is linearly ordered under $\leqq_I$, there is some $z\in M$ (in particular, the max $z_s$ under the order $z_s\leq z_t \iff (T,m_{z_s}) \leqq_I (T,m_{z_t})$) for which $(S,l)\leqq_m (T,m_z)$. But then, by our induction hypothesis, $\overline{(S,l)}\leqq \overline{(T,m_z)} \leqq \overline{(T,m)}$. \\
	\\
	Suppose instead that $L\leqq M$. The result is trivial if $M\in \{\omega,\omega^*\}$ so we suppose otherwise. Take $\mathcal{I}$ a collection of intervals in $M$ that is isomorphic to $M$ under the block ordering. Take an embedding $\iota:L\rightarrow \mathcal{I}$. By our induction hypothesis, for each $x\in L $,  $\exists f_x: \overline{(S,l_x)} \rightarrow \sum\limits_{y\in \iota(x)} \overline{(T,m_y)} \equiv \overline{(T,m)}$. So then, $f= \bigcup\limits_{x\in L} f_x$ witnesses $\overline{(S,l)} \leqq \sum\limits_{x\in L} \sum\limits_{y\in \iota(x)} \overline{(T,m_y)} \equiv \overline{(T,m)}$. \\
	\\
	\textbf{Case 2b:} $(T,m) = \sum\limits_{x\in M} (T,m_x)$ for $M\in \{\omega,\omega^*,C,C^*,\QQ\}$ and for each $x\in M$, $(T,m_x) \in \mathcal{U}_\alpha$ for some $\alpha<\beta $. $(S,l) = [L, (S_1,l_1),...,(S,l_n)]$ $L\in \{\omega,\omega^*\}$. The cases for $M\in \{\omega,\omega^* \}$ is the simplest and was done explicitly by Laver. For each $i \in \{1,...,n\}$, $(S_i,l_i) \leqq_m (T,m)$. In particular, by our induction hypothesis, $\overline{(S_i,l_{ij})} \leqq \overline{(T,m)}$. Take a collection of intervals $\mathcal{I}$ in $M$ block isomorphic to $n\times L$. Take $\iota: n\times L \rightarrow \mathcal{I}$, and for each $(i,j) \in n\times L$, take an embedding $f_{ij}: (S_i,l_{ij}) \rightarrow \iota(i,j)$. Then, $\bigcup\limits_{(i,j) \in n\times L} f_{ij}$ witnesses an embedding of $\overline{(S,l)} = \sum\limits_{j\in L} \sum\limits_{i=1}^n \overline{(S_i,l_{i,j}) }$ into $\overline{(T,m)}$. 
\end{proof}

\begin{lemma}
	If $\sum\limits_{x\in L} \phi_x = \Phi$ is a shuffle with $\phi_x<\Phi $ for $L\in \{C,C^*\}$, then $ \Phi \equiv \sum\limits_{x\in L} \psi_x$ where each $\psi_x$ is AI and $\{\psi_x:x\in L\}$ is totally ordered. 
\end{lemma}
\begin{proof}
	 By Theorem 3.7, we may assume each $\phi_x$ is AI. Ordering $\phi_x$ via $\phi_\alpha$, $\alpha<\omega_1$, we can instead analyze the sum $\sum\limits_{\alpha\in \omega_1} \phi_\alpha$. There exists $\psi_\alpha < \sum\limits_{\alpha\in \omega_1} \phi_\alpha$ such that $\alpha <\beta \Rightarrow \psi_\alpha \leqq \psi_\beta$ and $\forall \alpha \exists \beta $ for which $\phi_\alpha \leqq \psi_\beta $. It becomes clear that an $L$ shuffle of $ \psi_\alpha$ will suffice. 
\end{proof}
\begin{prop}
	For each $\phi \in \mathcal{H}$ $\exists (T,l) \in \mathcal{U}$ such that $\overline{(T,l)}\equiv \mathcal{U} $
\end{prop}
\begin{proof}
	Suppose the result is true for all $\psi<\phi$. Suppose $\phi = \sum\limits_{x\in L} \psi_x$ $L\in \{C,C^*,\QQ,\omega,\omega^*\}$ a shuffle. First suppose that $\{\psi_x: x\in L \}$ is or can be totally ordered under a change like in lemma 2.6. For each $x\in L$, find $(T_x,l_x)$ with $\overline{(T_x,l_x)} = \psi_x$. If $|L|>\aleph_0$, we can suppose $T_x = T_y$ for all $x,y\in L$. In this instance, we may also use the BQO property to further move down to a subsequence with $\{ (T_x,l_x) :x\in L\}$ linearly ordered under $\leqq_I$ and we'd be done. Suppose instead that $|L|=\aleph_0$. By Laver's covering theorem, we may suppose $(T_x,l_x)$ have at most $n$ treetops, have height bounded by $n$,  and is $\leq_m$ increasing. But then, there are only finitely many types of trees, meaning we can go down to a class of trees $(T_x,l_x)$ with $T_x=T_y$ for all $x,y\in L$. But then, again we may assume the trees are $\leq_I$ ascending like in the previous case, so that $(T,l) = \sum\limits_{x\in L} (T_x,l_x)$ does the job. \\
	\\
	If $ \{ \psi_x :x\in L\}$ cannot be totally ordered under $\leqq$, then $ L\in \{ \omega,\omega^* \}$. Then we may suppose $\phi = \sum\limits_{x\in L}\sum\limits_{i=1}^n \psi_{x,i} $, where for a fixed $i$, $\{ \psi_{x,i}: x\in L \}$ is totally ordered. Applying the previous case, we get $(T_i,l_i) = \sum\limits_{x\in L }(T_i,l_{i,x}) $ with $\overline{(T_i,l_{i,x})} =\psi_{i,x}$. But then, $(T,l) = [L: (T_1,l_1),...,(T_n,l_n)] $ does the trick. 
\end{proof}
\begin{theorem}
	If $(T,l) \in \mathcal{U}$ has $n$ treetops, then $\overline{(T,l)}$ has Ramsey degree bounded by $n$. 
\end{theorem}
\begin{proof}
	We do this by induction on $\mathcal{U}_\alpha$. It is true for $\alpha=0$. Suppose it is true for all $(T,l) \in \mathcal{U}_\alpha$ for $\alpha<\beta$. Let $(T,l)\in \mathcal{U}_\beta$. If $(T,l)$ is a shuffle, we are done by lemma 1.3. Suppose instead that $(T,l) = [\omega, (T_1,l_1),...,(T_r,l_r)]$ with $(T_i,l_i) \in \mathcal{U}_\alpha$ for some $\alpha<\beta $. Each $(T_i,l_i) = \sum\limits_{k\in \omega} (T_i,l_{i,k}) $. Note, $\overline{(T_i,l_i)} $ has $n_i$ tree tops and hence Ramsey degree bounded by $n_i$. $(T,l)$ has $n=\sum\limits_{i=1}^r n_i$ many tree tops $\overline{(T,l)} = \sum\limits_{k\in \omega} \sum\limits_{i=1}^r\overline{(T_i,l_{i,k})}  $. Let $M= \bigcup\limits_{k\in \omega} (\bigcup\limits_{i=1}^r M_{i,k} ) $ be a realization of this type. For $i\leq r$, we define $M_i =\bigcup\limits_{k\in \omega}M_{i,k}$. Note, each $M_i$ has Ramsey degree $n_i$ and $M_i$ forms a partition of $M$. Consequently, $M$ has Ramsey degree bounded by $n$.  
\end{proof}
\begin{corollary}{(PFA)}
	For any Aronszajn type $\phi$, there is an $n\in \omega$ such that $\forall k \in \omega $ $\phi \rightarrow (\phi)_{k,n}^1$. 
\end{corollary}
\begin{proof}
	The result is true if $\phi$ is not fragmented so suppose otherwise and let $\phi \in \mathcal{C}$.  By theorem 3.7, $\phi= \phi_1+...+\phi_k$ where each $\phi_i \in \mathcal{H}$. By proposition 4, for each $\phi_i$ there is $(T_i,l_i) $ such that $\overline{(T_i,l_i)} \equiv \phi_i$. By theorem 3.10, each $\overline{(T_i,l_i)}$ has finite Ramsey degree. Consequently, $\phi$ has finite Ramsey degree by lemma 2.1.  
\end{proof}

\end{document}